\documentclass[oneside,english]{amsart}

\usepackage{bookman} 

\usepackage{amssymb} 
\usepackage{amsmath,tikz,wasysym,multirow} 
\usepackage{amstext}
\usepackage{amsthm}
\usepackage{amssymb}

\newtheorem{Proposition}{Proposition}[section]
\newtheorem{Theorem}{Theorem}[section]
\newtheorem{example}{Example}[section]

\newtheorem{remark}{Remark}[section]

\newtheorem{Corollary}{Corollary}[section]
\newtheorem{Lemma}{Lemma}[section]
\newtheorem{question}{Question}[section]

\usepackage{babel}

\newcommand{\ideal}{\triangleleft}
\DeclareMathOperator{\ann}{Ann}
\DeclareMathOperator{\ass}{Ass}
\DeclareMathOperator{\zt}{ZT}

\newcommand{\ZZ}{\mathbb{Z}}
\newcommand{\UU}{\mathcal{U}}
\newcommand{\diam}[1]{{\rm diam}(#1)}

  \definecolor{polona}{rgb}{.2,.8,.2}
   \definecolor{nik}{rgb}{.1,.8,.9}

\begin{document}

\title{The total zero-divisor graph of commutative rings}

\date{\today}

\author{Alen \DJ uri\' c}
\address[A. \DJ uri\' c]{Faculty of Natural Sciences and Mathematics, University of Banja Luka, Mladena Stojanovi\' ca 2, 78000 Banja Luka, Bosnia and Herzegovina}
\email{alen.djuric@protonmail.com}

\author{Sara Jev\dj eni\' c}
\address[S. Jev\dj eni\' c]{Faculty of Natural Sciences and Mathematics, University of Banja Luka, Mladena Stojanovi\' ca 2, 78000 Banja Luka, Bosnia and Herzegovina}
\email{sarajevdjenic9@gmail.com}

\author{Polona Oblak}
\address[P.~Oblak]{Faculty of Computer and Information Science, University of Ljubljana, Ve\v cna pot 113, SI-1000 Ljubljana, Slovenia}
\email{polona.oblak@fri.uni-lj.si}

\author{Nik Stopar}
\address[N. Stopar]{Faculty of Electrical Engineering, University of Ljubljana, Tr\v za\v ska cesta 25, 1000 Ljubljana, Slovenia}
\email{nik.stopar@fe.uni-lj.si}

\thanks{The authors acknowledge the financial support from the Slovenian Research Agency  
(research core funding no. P1-0222) and the bilateral project no. BI-BA/16-17-025 funded by Slovenian Research Agency and Ministry of Civil Affairs, Bosnia and Herzegovina.}


\maketitle

\begin{abstract}
In this paper we initiate the study of the total zero-divisor graphs over commutative rings with unity. 
These graphs are constructed by both relations that arise from the zero-divisor graph and from the total graph
of a ring. We characterize Artinian rings with the connected total zero-divisor graphs and give their
diameters. Moreover, we compute major characteristics of the total zero-divisor graphs of 
the ring $\ZZ_m$ of integers modulo $m$ and prove that the total zero-divisor graphs of
$\ZZ_m$ and $\ZZ_n$ are  isomorphic if and only if $m=n$.
\end{abstract}

\bigskip
\section{Introduction}
\bigskip

Recently, the interplay between properties of algebraic structures and their relationship 
graphs has been studied extensively. To further the understanding of the structure of  zero-divisors in 
semigroups, several various graphs were introduced.

In 1988, Beck \cite{MR944156} introduced  the coloring properties of a graph,
whose vertices were all the elements of the ring and two vertices were adjacent if their product was 0. This 
definition was simplified in 1999 by Andreson and Livingston \cite{MR1700509} to the zero-divisor graph.
The vertices of the  zero-divisor graph are all nonzero zero-divisors 
and distinct vertices $x$ and $y$ are adjacent if and only if $xy=0$. 
Compared with the Beck's graph, they omitted  0 and all vertices which are not zero-divisors in the graph, and thus the properties of the zero-divisors in the ring were more clearly reflected.
Zero-divisor graphs were largely studied in the last two decades (see e.g. \cite{MR3675901, MR2932598} for survey papers). 
In 2002, Mulay \cite{MR1915011} defined a compressed zero-divisor graph, in which the vertices are the equivalence classes of zero-divisors of the ring. Compressed graphs are smaller hence easier to investigate than
the zero-divisor graphs. All those graphs were extensively studied over structures with one or two operations.

In \cite{MR2441996} Anderson and Badawi introduced the total graph of a commutative ring $R$ as the graph with vertices of all elements of $R$ and with edges  $x \sim y$ 
for distinct $x, y \in R$ if $x + y$ is a zero-divisor in $R$. 
The total graph of the ring engages both ring operations instead of studying only multiplication as in the zero-divisor graph and therefore reflects more structure of the subset of zero-divisors in a given ring.

In this paper we initiate the study of both relationships to reveal more information about the commutative 
rings. Denote the set of all zero divisors of $R$ by $Z(R)$. The \emph{total zero-divisor graph} $\zt\left(R\right)$ of a commutative unital ring $R$ is a 
simple graph whose
vertex set is the set of all nonzero zero divisors of $R$, and where distinct vertices 
$u$ and $v$ are adjacent if and 
only if 
$$uv=0 \, \text{ and }\, u+v\in Z\left(R\right).$$

In Section \ref{connectedness} we investigate the connectedness of a total zero-divisor graph of a
commutative unital ring. In Theorem \ref{thm:iff} we characterize Artinian rings with a connected total 
zero-divisor graph. On the contrary to the zero-divisor graph or to the total graph, this characterization
highly depends on the associated prime ideals of the ring. Moreover, we provide the diameter of the graph (see Theorem \ref{thm:diam}), which depends on the maximal ideals of the ring.

In Section \ref{Zm} we compute the main characteristics of the total zero-divisor graph of 
the ring  $\ZZ_m$ of integers modulo $m$, such as girth, degrees, chromatic number, domination
number and metric dimension. Since $\ZZ_m$ is a prototype for certain classes of rings, we believe 
some of the results might be generalized. In Corollary \ref{cor:deltas--m}  we show that the total 
zero-divisor graphs of
$\ZZ_m$ and $\ZZ_n$ are  isomorphic if and only if $m=n$. Moreover, we characterize all 
acyclic total zero-divisor graphs in Corollary \ref{acyclic}.

\bigskip
\section{Preliminaries}
\bigskip

Throughout the paper, let $R$ be a  commutative  ring with unity. Let $Z\left(R\right)$ denote the
set of zero-divisors in $R$ and $\UU\left(R\right)$ the set of invertible elements in $R$. If $I$ is a
nilpotent ideal in $R$, we call the smallest positive integer $\ell$ such that the product of any $\ell$ 
 elements of $I$ is $0$, a \emph{nilindex} of $I$.

For an $R$-module $M$ and a subset $S \subseteq M$ we denote by 
$$\ann_R S=\{r \in R; \;rs=0 \text{ for all } s \in S\}$$  
the \emph{annihilator of $S$ in $R$}. We will omit index assuming that it is $R$ if not stated otherwise.
We say that a prime ideal $P$ of $R$ is an \emph{associated prime} of $M$ if $P=\ann(m)$ for some $m \in M$.
The set of all associated primes of $M$ is denoted by $\ass(M)$.
 If $R$ is  Noetherian  and $M$ a finitely generated $R$-module, then $\ass(M)$ is finite 
(see e.g. \cite[Proposition (7.G)]{MR575344}).

We say that an element $r \in R$ is a zero-divisor for $M$ if $rm=0$ for some non-zero $m \in M$.
Recall that if $R$ is Noetherian ring and $M$ a finitely generated non-zero $R$-module, then the set 
of zero-divisors for 
$M$ is the union of all the associated primes of $M$ \cite[Proposition (7.B)]{MR575344}.
In the special case where we treat $R$ as an
$R$-module, we have
\begin{equation}\label{Zunion}
Z(R)=\bigcup_{P \in\ass(R) }P.
\end{equation}

Our notation for the graphs is the following. For a graph $G=(V(G),E(G))$ we denote its order by 
$|G|=|V(G)|$.  
The sequence of edges $x_0 \sim x_1 \sim \ldots \sim x_{k-1} \sim x_{k}$ in a graph is called 
\emph{a path of length $k$}. The distance between two vertices is the length of the shortest path between them and the diameter $\diam{G}$ is the largest distance between any two vertices of the graph. 
The \emph{open neighbourhood} of a vertex $v$ is denoted by $N_G(v)$ and consists of all vertices at 
distance 1 from vertex $v$ and the \emph{closed neighbourhood} of a vertex $v$ is denoted by 
$N_G[v]=N_G(v) \cup \{v\}$ and consists of all vertices at distance at most 1 from vertex $v$.
When stated without any qualification, a neighbourhood is assumed to be open.
 
 A path $x_0 \sim x_1\sim \ldots \sim x_{k} \sim x_{0}$ is called a \emph{cycle}.  A graph with a cycle will be called a \emph{cyclic} graph and \emph{acyclic} otherwise.
A \emph{Hamiltonian cycle} of a graph $G$ is a cycle that contains every vertex of $G$. A graph is 
called \emph{Hamiltonian} if it contains a Hamiltonian cycle. A graph G is \emph{Eulerian} if it contains a cycle that consists of all the edges of G.
A complete graph on $n$ vertices will be denoted by $K_n$ and a path with $n$ vertices will be denoted by $P_n$.


\bigskip
\section{Connectedness of the total zero-divisor graph}\label{connectedness}
\bigskip

The zero-divisor graphs are always connected with the diameter at most three \cite{MR1700509}. On the other hand,
 the total graph of a commutative ring is not connected if the set of zero-divisors  forms an ideal  \cite{MR2441996}. We will see that the 
 connectedness of the total zero-divisor graph of the ring $R$  depends on the 
existence of a maximal associated prime $P \ideal R$, such that $P \cap \ann P=\{0\}$. 
Note that for example in the ring $\mathbb{Z}_{6}$ every prime ideal intersects its annihilator trivially, but in 
$\mathbb{Z}_{36}$ no prime ideal intersects its annihilator trivially.

First we prove the following lemma, which will give us a necessary condition for the graph $\zt(R)$ to be
connected. We say that $P$ is a maximal associated prime of $R$ if it is maximal among all associated primes of $R$.

\begin{Lemma}\label{E-notconn}
Let $R$ be a  Noetherian ring  with unity and  assume $Z(R)\neq \{0\}$. If there exists a maximal associated prime 
$P$ of $R$ such that $P \cap \ann P=\{0\}$ then
the graph $\zt(R)$ is not connected.
\end{Lemma} 

\begin{proof} Let $P$ be a maximal associated prime of $R$ such that $P \cap \ann P=\{0\}$. 
Since $R$ is Noetherian, the set $\ass(R)$ is finite. 
Due to the maximality of $P$, the prime avoidance lemma (see e.g. \cite[Lemma~3.3]{MR1322960}) implies that the set $P \setminus (\bigcup_{Q \in \ass(R)\backslash\{P\}} Q)$ is nonempty. 
If $P \setminus (\bigcup_{Q \in \ass(R)\backslash\{P\}} Q)=\{0\}$, then $\ass(R)=\{P\}$ and $P=0$. By \eqref{Zunion} this implies $Z(R)=\{0\}$, a contradiction. Hence $P \setminus (\bigcup_{Q \in \ass(R)\backslash\{P\}} Q)\neq \{0\}$.

Take any $0 \neq x \in P \setminus (\bigcup_{Q \in (\ass(R))\backslash\{P\}} Q)$. By \eqref{Zunion}, $x$ is a nonzero zero-divisor. We will show that $x$ is an isolated vertex of $\zt(R)$.
Let $y \neq x$ be a nonzero zero-divisor and suppose that $x$ and $y$ are adjacent in $\zt(R)$.
By \eqref{Zunion} this implies $x+y \in Q$ for some $Q \in \ass(R)$. Multiplying $x+y$ by $x$ and $y$ respectively and taking into account that $xy=0$ gives us $x,y \in Q$, since $Q$ is prime. By definition of $x$ we have $Q=P$, so $y \in P$.

If we view the ideal $(y)$ as an $R$-module, then $x$ is a zero-divisor for $(y)$, so $x \in \widehat{P}$ for some associated prime $\widehat{P}$ of $(y)$. In particular, $\widehat{P}=\ann(z)$ for some $z \in (y) \subseteq P$. But $\widehat{P}$ is also an associated prime of $R$, hence $\widehat{P}=P$ by definition of $x$. We conclude that $z \in P \cap \ann P=0$, which implies $P=\ann (z)=R$, a contradiction since $P$ is prime. This shows that $x$ is an isolated vertex. 

To conclude the proof we need to show that $\zt(R)$ contains at least two vertices. Assume on the contrary that $\zt(R)$ has a single vertex $x$. Then $Z(R)=\{0,x\}$.
This implies that $\{0,x\}$ is the only associated prime of $R$, hence $P=\{0,x\}$. By assumption $P \cap \ann P=\{0\}$ which implies $x \notin \ann P$. Consequently $x^2 \neq 0$. But $x^2$ is a zero-divisor, so $x^2=x$ which implies $x(1-x)=0$. Thus $1-x$ is a non-zero zero-divisor as well. Hence $1-x = x$ which implies  $2x=1$, a contradiction since $x$ is a zero-divisor.
\end{proof} 

In the next example we show that Lemma \ref{E-notconn} is not valid for non-Noetherian rings.

\begin{example}
Let $R$ be the ring of all real valued functions on $[0,1]$. To show that $R$ is not
Noetherian, let $I_{S}\subseteq R$ denote a subring that contains functions
which vanish on a subset $S\subseteq\left[0,1\right]$. For any infinite
decreasing chain $S_{0}\varsupsetneq S_{1}\varsupsetneq \cdots$ of
subsets of $\left[0,1\right]$, we get an infinite increasing chain
$I_{S_{0}}\varsubsetneq I_{S_{1}}\varsubsetneq\cdots$ of ideals of
$R$.

Now, define $f\left(x\right)=\begin{cases}
1, & x=0,\\
0, & x\neq0.
\end{cases}$ Let $P=\ann\left(f\right)=\left\{ g; \; g\left(0\right)=0\right\} $ and note that
$\ann  P=\left\{ h; \; h\left(x\right)=0 \text{ for all } x\neq0\right\} $, hence 
$P \cap \ann P=\{0\}$. Since $R/P\cong\mathbb{R}$, it follows that $P$ is maximal associated prime.
Since the zero-divisors in $R$ are the functions which are equal to 0 on a set with a nonzero measure, it is
possible to construct a path between any two vertices in  $\zt(R)$. Therefore, $\zt(R)$ is connected. 
\end{example}

We will prove that in the case of a  Noetherian and Artinian ring $R$, the graph $\zt(R)$ is connected if and only if
 $P \cap \ann P\ne\{0\}$ for all maximal associated prime ideals $P \ideal R$. First we need a few technical lemmas.



\begin{Lemma}\label{lemma:product}
Let $R$ be an Artinian 
ring and 
choose $a_1,a_2 \in R$, such that the ideals $Q_1=\ann(a_1)$ and 
$Q_2=\ann(a_2)$ are distinct prime ideals in $R$. It follows that $a_1a_2=0$.
\end{Lemma} 

\begin{proof}
Since $R$ is Artinian, we can suppose $R=R_1 \times R_2 \times \dots \times R_k$, for some local rings $R_1,R_2,...,R_k$. 

Let us first show that if $P$ is a prime ideal in $R$, then $P=P_1\times P_2 \times \dots \times P_r$,
where  $P_i$ is a prime ideal in $R_i$ for exactly one index $i$ and $P_j=R_j$ for all $j \neq i$. 
Assume the opposite and without loss of generality  suppose that $P_1 \neq R_1$ and $P_2 \neq R_2$. 
Choose $x=(p_1,1,0,0,...,0)$ and $y=(1,p_2,0,0,...0)$, where $p_i \in  P_i$. It follows that $xy\in P$, but neither $x$ nor $y$ is an element of the ideal $P$, a contradiction. 

Let $Q_1=\ann(a_1)$ have a proper prime ideal at $j_1$-th position and let $Q_2=\ann(a_2)$ have a 
proper prime ideal at position $j_2$, i.e. for $i=1,2$ we have 
$Q_i=R_1 \times \dots R_{j_i-1} \times P_i \times R_{j_i+1} \times \dots \times R_k$,  
where $P_i$ is a prime ideal in $R_{j_i}$.
Since $Q_k=\ann(a_k)$, $k=1,2$, it follows that
$a_1=(0,\ldots,0,p_1,0,\ldots,0)$ and $a_2=(0,\ldots,0,p_2,0,\ldots,0)$, where 
$P_k=\ann_{R_{j_k}}(p_k) \subseteq R_{j_k}$. 
If $j_1 \ne j_2$, then clearly $a_1a_2=0$. 
If $j_1=j_2$, assume without loss of generality that $j_1=j_2=1$. 
If neither $p_1$ nor $p_2$ is a zero-divisor in $R_1$, they are both invertible in $R_1$ and hence 
$Q_1=Q_2$, a contradiction. Otherwise, assume that $p_1$ is a zero-divisor.
Since $R_1$ is Artinian and local, we have $p_1^n=0 \in P_2$ for some positive
integer $n$. Since $P_2$ is prime, we have $p_1 \in P_2=\ann(p_2)$, and therefore $p_1p_2=0$
and   $a_1a_2=0$. Similarly we argue that $a_1a_2=0$ in the case $p_2$ is a zero-divisor.
\end{proof}

\begin{Lemma}\label{lemma:clique}
Let $R$ be an Artinian ring and $I$ ideal generated by $$S=\bigcup\limits_{P \in \ass(R)}{P\cap \ann P}.$$
Then the subgraph of $\zt(R)$ induced by vertices in $I$ is a complete graph.

\end{Lemma}

\begin{proof}
Let us first show that $S^2=0$. Take arbitrary  
$s_1 \in P_1\cap \ann P_1$ and $s_2 \in P_2\cap \ann P_2$ for some 
$P_1,P_2 \in \ass(R)$. Thus there exist $a_1,a_2 \in R$, such that $P_i=\ann (a_i)$, $i=1,2$, 
and so  $s_i \in \ann (a_i) \cap \ann (\ann (a_i))$. It follows that $s_i^2=0$ and 
$\ann(a_i) \subset \ann(s_i)$, $i=1,2$.

If $P_1=P_2$, then $s_1,s_2 \in \ann (a) \cap \ann (\ann (a))$ and therefore $s_1s_2=0$.
If $P_1$ and $P_2$ are distinct, it follows by Lemma \ref{lemma:product} that $a_1a_2=0$.
This implies that $a_1\in \ann (a_2)\subset \ann (s_2)$, so $a_1s_2=0$. Thus,
$s_2 \in \ann (a_1)\subset \ann (s_1)$ and so $s_1s_2=0$. 

If  $I$ is the ideal, generated by $S$, notice that $S^2=0$ gives us $I^2=0$. 
Hence, $xy=0$  and  $(x+y)^2=0$ for any $x,y \in I$ and thus the vertices from $I$ form
a clique in $\zt(R)$.
\end{proof}

%
%

\begin{Lemma}\label{notE-conn}
Let $R$ be an Artinian ring. If $P \cap \ann P\ne \{0\}$ for each maximal associated prime ideal
$P$ of $R$, then $\zt(R)$ is connected with $\diam{\zt(R)}\leq 3$.
\end{Lemma}

\begin{proof}
Take any nonzero $x\in Z(R)$. By \eqref{Zunion} there exists $P'\in \ass(R)$, such that
$x \in P'$. Take maximal $P\in \ass(R)$ that includes $P'$. 
By assumption $P \cap \ann P\ne \{0\}$ there exists $z_x \in P \cap \ann P$, which implies $xz_x=0$.
Since $P$ is an ideal, we have $x+z_x \in P \subset Z(R)$, and thus $x \sim z_x$ is an
edge in $\zt(R)$. By Lemma \ref{lemma:clique} vertex $z_x$  is a vertex of a clique in 
$\zt(R)$. Therefore, for any distinct nonzero $x,y\in Z(R)$, there exists
a path $x \sim z_x \sim z_y \sim y$ in $\zt(R)$ and the result follows.
\end{proof}

Recall that a ring $R$ is Artinian if and only if it is Noetherian and all the prime ideals in $R$ are maximal \cite[Theorem~2.14]{MR1322960}. Hence, as a corollary of Lemma \ref{E-notconn} and Lemma \ref{notE-conn}, we have the following.

\begin{Theorem}\label{thm:iff}
For an Artinian ring $R$, graph $\zt(R)$ is connected if and only if $P \cap \ann P \neq \{0\}$ for all maximal associated  prime ideals $P \in \ass(R)$.
\end{Theorem}

In the case graph $\zt(R)$ is connected, it is possible to classify $R$ with respect to diameter of
$\zt(R)$.

\begin{Theorem}\label{thm:diam}
Let $R$ be an Artinian ring, such that $\zt(R)$ is connected and not empty. Then
\begin{enumerate}
\item If $R$ is not local, then $\diam{\zt(R)}=3$,
\item If $R$ is local and $\mathfrak{m} \ideal R$ is its unique maximal ideal with nilindex $n \geq 3$, then $\diam{\zt(R)}=2$,
\item If $R$ is local and $\mathfrak{m} \ideal R$ is its unique maximal ideal such that $\mathfrak{m}^2=0$, $\mathfrak{m}\neq\{0\}$, then $\diam{\zt(R)}=1$.
\end{enumerate}
\end{Theorem}
\begin{proof}
If $R$ is  an Artinian ring that is not local, then $R=R_1\times R_2 \times \dots \times R_n$, where $k\geq 2$ and $R_1,R_2,...,R_k$ are local rings. 
Denote $x=(0,1,1,\ldots,1)$ and $y=(1,0,0,\ldots,0)$. Note that $x$ and $y$ are not adjacent, since $x+y=(1,1,\ldots,1) \notin Z(R)$. Furthermore, if $(a_1,a_2,\ldots,a_k) \in N(x) \cap N(y)$, then $a_i=0$, $i=1,2,\ldots,k$, 
but $(0,0,\ldots,0) \notin V(\zt(R))$. Therefore 
$d(x,y) \geq 3$ and by Theorem \ref{notE-conn} it follows that $\diam{\zt(R)}=3$.

Suppose now $R$ is local Artinian with unique maximal ideal $\mathfrak{m}$. Then $Z(R)=\mathfrak{m}$ and so $V(\zt(R))=\mathfrak{m} \setminus \{0\}$.
In the case $\mathfrak{m}^2=\{0\}$, we have $x+y \in Z(R)$ and $xy=0$ for any $x,y \in \mathfrak{m}$. Therefore $\diam{\zt(R)}=1$ and nonzero elements of $\mathfrak{m}$ form a clique. 

Suppose now, there exists nonzero $u\in \mathfrak{m}^{n-1}$ and $\mathfrak{m}^{n}=0$ for some $n \geq 3$. 
Thus, for any nonzero  $x \in \mathfrak{m} \setminus \{u\}$, we have $xu=0$ and $x+u \in \mathfrak{m}$, so $x$ is adjacent to $u$ in $\zt(R)$. 
Hence, $\diam{\zt(R)} \leq 2$. We would like to 
show that $\zt(R)$ is not complete graph. 
Since $n \geq 3$, there exist $v,w \in \mathfrak{m}$ such that $vw \neq 0$. If $v \ne w$, then $v \nsim w$ and thus  $\diam{\zt(R)} = 2$. 
Otherwise, if $vw=0$ for every two distinct $v,w \in \mathfrak{m}$, there exists $u \in \mathfrak{m}$ such 
that $u^2\neq0$. This implies $u(u+w)=u^2\neq0$ for every $w \in \mathfrak{m}$. 
If $u+w$ is nonzero and distinct from $u$, it implies $u \nsim u+w$, hence $\diam{\zt(R)} = 2$. 
In the case $u+w=0$ for some $w \in \mathfrak{m}$, we have $u^2=0$,
a contradiction. Therefore, for every $w \in \mathfrak{m}$ we have $u+w=u$, which implies $\mathfrak{m}=\{0,u\}$ and so  $u^2=u$. This contradicts the fact that $u^n=0$.
\end{proof}

We can easily apply Theorems \ref{thm:iff} and \ref{thm:diam} to the case $R=\ZZ_m$.

\begin{Corollary}\label{prop:connected--iff}
The total zero-divisor graph  $\zt\left(\ZZ_{m}\right)$ 
is connected for $m=p_{1}^{m_{1}}p_{2}^{m_{2}}\cdots p_{n}^{m_{n}}$, where 
$p_{1}<p_{2}<\cdots<p_{n}$ are prime numbers, if and only if  $m_{i}\geq2$ for $i=1,2,\ldots,n$.

Moreover, if $\zt\left(\ZZ_{m}\right)$
is connected, then
\begin{enumerate}
\item $\diam{\zt\left(\ZZ_{m}\right)}=3$ if 
$n\geq2$,
\item $\diam{\zt\left(\ZZ_{m}\right)}=2$ if 
$m=p^{k}$, $k\geq3$,
\item \label{complete} $\diam{\zt\left(\ZZ_{m}\right)}=1$ 
and $\zt\left(\ZZ_m\right)= K_{p-1}$ if 
$m=p^{2}$.  
\end{enumerate}
\end{Corollary}

\bigskip
\section{The total zero-divisor graph of $\ZZ_m$}\label{Zm}

\bigskip

If not stated differently, an integer $m$ will be factorized as 
$$m=p_{1}^{m_{1}}p_{2}^{m_{2}}\cdots p_{n}^{m_{n}},$$
where $p_{1}<p_{2}<\cdots<p_{n}$ are prime numbers.
%

We will often make a use of the following lemma, which follows  from Corollary \ref{prop:connected--iff}
in the case $R=\ZZ_m$.

\begin{Lemma}\label{thm:cliqueZm}
If the total zero-divisor graph $\zt\left(\ZZ_{m}\right)$ is connected, the vertices in  $\zt\left(\ZZ_{m}\right)$ of the form 
 $$r \prod_{i=1}^n p_i^{\left\lceil \frac{m_i}{2}\right\rceil},$$ 
where $1 \leq r \leq p_1^{\left\lfloor \frac{m_1}{2}\right\rfloor} \dots p_n^{\left\lfloor \frac{m_n}{2}\right\rfloor}-1$, 
form a clique  of size 
$p_1^{\left\lfloor \frac{m_1}{2}\right\rfloor} \dots p_n^{\left\lfloor \frac{m_n}{2}\right\rfloor}-1$.
\end{Lemma}

From this fact we can compute some other parameters of the total zero-divisor graph.

\begin{Lemma}
Graph $\zt\left(\ZZ_{m}\right)$ is acyclic if $m$ is  equal to $2^{2}$, $3^{2}$,  $2^{3}$, $m=2^{2}p$ 
or $m=pq$, for any distinct prime numbers $p,q$. 
Otherwise, $\textrm{girth}\zt\left(\ZZ_{m}\right)=3$.
\end{Lemma}

\begin{proof} By Lemma \ref{thm:cliqueZm} we have that $\textrm{girth}\zt\left(\ZZ_{m}\right)=3$ if the graph is connected and $m$ distinct from $2^2$, $2^3$, $3^2$ and $3^3$.
Note that $\zt\left(\ZZ_{2^2}\right)=K_1$,  $\zt\left(\ZZ_{3^2}\right)=K_2$, $\zt\left(\ZZ_{2^3}\right)=P_3$, which are all acyclic. In $\zt\left(\ZZ_{3^3}\right)$ we have
$3 \sim 9 \sim 18 \sim 3$ and so $\textrm{girth}\zt\left(\ZZ_{3^3}\right)=3$.

If $p$ and $q$ are distinct primes, $m=pq$, then $\zt\left(\ZZ_{pq}\right)$ is an empty graph with $p+q-2$ vertices.
In the case $m=2^{2}p$, observe that every edge in the graph is incident with $2p$, hence  $\zt\left(\ZZ_{2^2p}\right)$ is acyclic.
However, if $m=p^{m_p}q^{m_q}$, where  $m_q \geq 2$, and  $m\ne 2^{2}p$, then $q^{m_{q}-1} \sim \frac{m}{q} \sim \left(q-1\right)\frac{m}{q} \sim q^{m_{q}-1}$ if $q>2$
and  $2^{m_{q}-1} \sim \frac{m}{2} \sim \frac{m}{2^{2}} \sim 2^{m_{q}-1}$ if $q=2$. In both cases, $\textrm{girth}\zt\left(\ZZ_{p^{m_p}q^{m_q}}\right)=3$.
If $n\geq3$, then $\frac{m}{p_i}  \sim \frac{m}{p_j}$ for all $1 \leq i,j \leq n$ and so $\textrm{girth}\zt\left(\ZZ_{m}\right)=3$.
\end{proof}

\begin{Corollary}\label{acyclic}
 The only acyclic total zero-divisor graphs of rings $\ZZ_m$ are $P_1$, $P_2$, $P_3$, $K_{1,2p-2}\cup 2K_1$ and $(p+q-2) K_1$ for any distinct prime numbers $p$ and $q$.
\end{Corollary}

\begin{question}
 Theorems \ref{thm:iff} and \ref{thm:diam} give us some necessary conditions for a graph $G$ to be the total zero-divisor graph of a Noetherian and Artinian ring $R$. 
 Is it possible to classify all acyclic total zero-divisor graphs of such rings $R$? Or, more generally, is it possible to classify all graphs that are the
 total zero-divisor graph of a Noetherian and Artinian ring $R$?
\end{question}

\bigskip
\subsection{Associated elements in $\ZZ_m$}

We say that two elements $a,b \in \ZZ_{m}$ are \emph{associated} if $b=au$ for some invertible element 
$u \in \UU(\ZZ_{m})$. We name the set of all $\{au; \; u \in  \UU(\ZZ_m)\}$ the associatedness class of $a$ in $\ZZ_m$.
Note that for associated zero-divisors $a$ and $b$, $N(a)\setminus \{b\}=N(b)\setminus \{a\}$ holds, hence associated vertices are indistinguishable. 
While in the zero-divisor graph the only way for vertices to be indistinguishable is to be associated, in the total zero-divisor graph there are vertices that are indistinguishable even though they are not associated. Namely, $p_i ^{m_i}$ and $p_i^{m_i-1}$ where $m_i \geq 2$, are such vertices.

We use the notation $w|_{m}v$ to denote that $w$ is a divisor of $v$ in $\ZZ_{m}$, i.e. there exists $x \in \ZZ_{m}$ such that $wx \equiv v$ modulo $m$.
The symbol "$\equiv$'' denotes congruence modulo $m$, i.e. equality in ring $\ZZ_{m}$.

\begin{Lemma}
\label{lem:associated--divisor}Every associatedness class of $\ZZ_{m}$
contains exactly one positive divisor of $m$.\end{Lemma}
\begin{proof}
We are going to show that for an arbitrary $a=p_{1}^{\ell_{1}}p_{2}^{\ell_{2}}\cdots p_{n}^{\ell_{n}}s_{a}\in\ZZ_{m}$,
$s_a \in \UU(\ZZ_{m})$, there is a divisor of $m$ which is associated to $a$. Let us denote
$b\left(a\right)=\left\{ i; \; \ell_{i}>m_{i}\right\} \subseteq \left\{ 1,2,\ldots,n\right\}$.
Then, we have
\begin{eqnarray*}
a & \equiv & a-m \prod_{  j\notin b\left(a\right)} p_{j} =  
p_{1}^{\ell_{1}}p_{2}^{\ell_{2}}\cdots p_{n}^{\ell_{n}}s_a-m\cdot\prod_{j \notin b\left(a\right)}p_{j}=\\
 & = & \prod_{  i\in b\left(a\right)} p_{i}^{m_{i}}\cdot
 \prod_{j\notin b\left(a\right)}p_{j}^{\ell_{j}}\left(
  \prod_{i\in b\left(a\right)}p_{i}^{\ell_{i}-m_{i}}\cdot s_a-\prod_{j\notin b\left(a\right)}p_{j}^{m_{j}-\ell_{j}+1}\right)\\
 & = & v\cdot u,
\end{eqnarray*}
where 
$$
u=\left(\prod_{i\in b\left(a\right)}p_{i}^{\ell_{i}-m_{i}}\cdot s_a-\prod_{j\notin b\left(a\right)}p_{j}^{m_{j}-\ell_{j}+1}\right)
$$
 and 
$$
v=\prod_{i\in b\left(a\right)}p_{i}^{m_{i}}\cdot\prod_{j\notin b\left(a\right)}p_{j}^{\ell_{j}}.
$$
Observe that $u$ is invertible because if $u$ had a $p_{k}$ as a divisor,
$k$ would have to be in both $b\left(a\right)$ and $\left\{ 1,2,\ldots,n\right\} \setminus b\left(a\right)$.
By definition, $v$ is a divisor of $m$ associated to $a$.

If two positive divisors of $m$, say $v$
and $w$, were associated, we would have $v|_{m}w$ and $w|_{m}v$.
Thus $v|(am+w)$ and $w|(bm+v)$ for some $a,b\in\ZZ$, and hence $v|w$ and $w|v$. Therefore $v=w$, which proves the uniqueness of a divisor of $m$ corresponding to associatedness class of $v$ and $w$.
\end{proof}

\begin{Corollary}\label{cor:dominating--vertex}
For an arbitrary zero-divisor $a=p_{1}^{\ell_{1}}p_{2}^{\ell_{2}}\cdots p_{n}^{\ell_{n}}s_{a}$ in $\ZZ_{m}$,
there exists $\frac{m}{p_{k}}$, $1 \leq k \leq n$,
such that $a|_{m} (\frac{m}{p_{k}})$.
\end{Corollary}

\begin{proof}
For a vertex $v$ as found  in the proof of Lemma \ref{lem:associated--divisor},
there exists $k \notin b\left(a\right)$, 
such that $v_{k}<m_{k}$, where $ v_k$ is the the greatest integer such that $p_i^{v_k}$ divides $v$. So $v|\left(\frac{m}{p_{k}}\right)$ and hence $a |_{m}\left(\frac{m}{p_{k}}\right)$.\end{proof}

In the case $m=p^2$, we have by Corollary \ref{prop:connected--iff} that $\zt(\ZZ_m)=K_{p-1}$ and thus the maximal degree of a vertex in the graph
$\Delta$ and the minimal degree $\delta$ are both equal to $p-2$. The next lemma gives us $\Delta$ and $\delta$ for other connected graphs.

\begin{Lemma}\label{prop:max--min--degree}
Let $\zt\left(\ZZ_{m}\right)$
be connected but not complete graph. Then 
$$\Delta\left(\zt\left(\ZZ_{m}\right)\right)=\textrm{deg}\left(\frac{m}{p_{i}}\right)=\frac{m}{p_1}-2 \,
\text { and } \delta\left(\zt\left(\ZZ_{m}\right)\right)=\textrm{deg}\left(p_{1}\right)=p_{1}-1.$$
\end{Lemma}

\begin{proof}
Consider the vertex $q_i=\frac{m}{p_{i}}$ for any $i\in\left\{ 1,2,\ldots,n\right\} $ and its closed neighbourhood  
$N\left[q_i\right]=\left\{ p_{i}r; \; 1\leq r\leq \frac{m}{p_{i}}-1\right\} $.
We have $\textrm{deg}\left(q_i\right)=\frac{m}{p_{i}}-2$.
By Corollary \ref{cor:dominating--vertex}, for an arbitrary vertex
$q$ in $\zt\left(\ZZ_{m}\right)$ there exists $k\in\left\{ 1,2,\ldots,n\right\} $ such that $q|_{m}\left(\frac{m}{p_{k}}\right)$.
Hence, $N\left[q\right] \subseteq N\left[q_k\right]$ and so we have 
$\textrm{deg}\left(q\right)\leq\textrm{deg}\left(q_k\right)=\frac{m}{p_{k}}-2\leq \frac{m}{p_1}-2$.

Now, consider the vertex $p_{i}$ for any $i\in\left\{ 1,2,\ldots,n\right\} $.
Its open neighbourhood is equal to  $N\left(p_{i}\right)=\left\{ r\cdot \frac{m}{p_{i}}:1\leq r\leq p_{i}-1\right\} $ 
and hence $\textrm{deg}\left(p_{i}\right)=p_{i}-1$.
%
For an arbitrary vertex $v=p_{1}^{v_{1}}p_{2}^{v_{2}}\cdots p_{n}^{v_{n}}s_{v}\in V\left(\zt\left(\ZZ_{m}\right)\right)$,
there is $l\in\left\{ 1,2,\ldots,n\right\} $ such that $v_{l}\geq1$.
 Hence, $N\left[p_{l}\right] \subseteq N\left[v\right]$ and thus $\textrm{deg}\left(v\right)\geq\textrm{deg}\left(p_{l}\right)=p_{l}-1\geq p_{1}-1$.
 \end{proof}

 If $G$ is a connected but not complete graph, isomorphic to $\zt\left(\ZZ_{m}\right)$ for some $m$, then 
by Lemma \ref{prop:max--min--degree} we have $m=\left(\Delta\left(G\right)+2\right)\left(\delta\left(G\right)+1\right)$. Therefore, we have the following property of the total zero-divisor graph.

\begin{Corollary}\label{cor:deltas--m}
If  $\zt\left(\ZZ_{m}\right) \cong \zt\left(\ZZ_{n}\right)$ is a connected graph for some integers $m,n$, then $m=n$.
\end{Corollary}

 \begin{Corollary} \label{prop:regular--complete}
A connected graph $\zt\left(\ZZ_{m}\right)$ is regular if and only if it is complete.
\end{Corollary}

\begin{proof}
Assume $\zt\left(\ZZ_{m}\right)$
be connected but not complete graph. If it is regular, we have $\Delta\left(\zt\left(\ZZ_{m}\right)\right)=\delta\left(\zt\left(\ZZ_{m}\right)\right)$.
By Lemma \ref{prop:max--min--degree} it follows that $\frac{m}{p_1}-2=p_{1}-1\Rightarrow \frac{m}{p_1}=p_{1}+1$.
By Corollary \ref{prop:connected--iff}, connectedness of $\zt\left(\ZZ_{m}\right)$
implies $m_1 \geq 2$ and thus $p_{1}|\left(\frac{m}{p_1}\right)$, a contradiction. So, 
$\zt\left(\ZZ_{m}\right)$ is not regular graph. Since every complete graph is regular, we have the equivalence.
\end{proof}

\begin{Corollary}
For a pair of integers $\left(x,y\right)$, there is a positive integer
$m$ such that
\begin{enumerate}
\item[(a)] graph $\zt\left(\ZZ_{m}\right)$ is connected but not complete,
\item[(b)] $\Delta\left(\zt\left(\ZZ_{m}\right)\right)=x$ and
$\delta\left(\zt\left(\ZZ_{m}\right)\right)=y$;
\end{enumerate}
if and only if the following conditions hold:
\begin{enumerate}
\item $x>y>0$,
\item $y+1$ is the least prime divisor of $x+2$,
\item if $q|\left(x+2\right)\left(y+1\right)$ then $q^{2}|\left(x+2\right)\left(y+1\right)$
for all prime $q$.
\end{enumerate}
\end{Corollary}

\begin{proof}
Assume that there are positive integers $m,x,y$ such that graph $\zt\left(\ZZ_{m}\right)$
satisfies conditions (a) and (b). By Lemma \ref{cor:deltas--m}, $m$ is uniquely determined as
$m=\left(x+2\right)\left(y+1\right)$. Since $\zt\left(\ZZ_{m}\right)$ is connected, Corollary \ref{prop:connected--iff} implies (3) and moreover that $y>0$. 
In particular, $p_{1}|\frac{m}{p_1}$ and $p_{1}$ is the smallest prime divisor of $\frac{m}{p_1}$.
Since $\zt\left(\ZZ_{m}\right)$ is not complete, Proposition \ref{prop:regular--complete} implies $x>y$ and thus (1) and (2).

Conversely, assume that $\left(x,y\right)$ is a pair of integers
such that conditions (1), (2) and (3) hold. Take $m=\left(x+2\right)\left(y+1\right)$. Corollary \ref{prop:connected--iff} and (3) imply that the graph $\zt\left(\ZZ_{m}\right)$
is connected. By (2), $y+1$ is the least prime divisor of $m$ and thus by Proposition \ref{prop:max--min--degree} it follows that $\Delta\left(\zt\left(\ZZ_{m}\right)\right)=\frac{m}{y+1}-2=x$
and $\delta\left(\zt\left(\ZZ_{m}\right)\right)=\left(y+1\right)-1=y$. Moreover, by (1),  $\zt\left(\ZZ_{m}\right)$
is not regular thus by \ref{prop:regular--complete} not complete.
\end{proof}

\bigskip
\subsection{Colorings of $\zt\left(\ZZ_{m}\right)$}

In this subsection we will investigate two different colorings of the total zero-divisor graph. 

The smallest number of colors needed to color the vertices of graph $G$, such that any two adjacent vertices have different colors,  is called the \emph{chromatic number} of $G$ and is denoted by $\chi(G)$.

Every nonempty bipartite graph has the chromatic number  equal to 2. 
Thus by Corollary \ref{acyclic} we have that the $\chi\left(\zt(\ZZ_{m})\right)=2$ if and only if
$m$ is $2^{2}$, $2^{3}$, $3^{2}$, $m=2^{2}p$ 
or $m=pq$, for any distinct prime numbers $p$ and $q$. 
In the following proposition we constructively  find the chromatic number of any total zero-divisor graph $\zt\left(\ZZ_{m}\right)$.

\begin{Proposition}\label{prop:chromatic}
In the case $\zt\left(\ZZ_{m}\right)$ is connected, let $o(m)$ denote the number of odd exponents $m_i$. Then 
the chromatic number of graph $\zt\left(\ZZ_{m}\right)$
is equal to
$$\chi\left(\zt\left(\ZZ_{m}\right)\right)=p_{1}^{\left\lfloor\frac{m_1}{2}\right\rfloor }\cdots p_{n}^{\left\lfloor\frac{m_n}{2}\right\rfloor }+ o(m)-1.$$

If $\zt\left(\ZZ_{m}\right)$ is not connected and not empty, then let  $I \subseteq \{1,\ldots, n\}$ be the set of all indices $i$ such that $m_i\geq 2$ and let $J=\{1,\ldots,n\} \setminus I$. Then the chromatic number of  
$\zt\left(\ZZ_{m}\right)$
is equal to
$$\chi\left(\zt\left(\ZZ_{m}\right)\right)=\prod_{i \in I} \left\lceil\frac{m_i}{2}\right\rceil + |J|.$$  
\end{Proposition}

\begin{proof}
Suppose first that $\zt\left(\ZZ_{m}\right)$ is connected. By Lemma \ref{thm:cliqueZm} we need at least 
$ \prod\limits_{i=1}^n p_{i}^{\left\lfloor\frac{m_i}{2}\right\rfloor } -1$ 
colors  to color the vertices of the form 
 $r \prod\limits_{i=1}^n p_i^{\left\lceil \frac{m_i}{2}\right\rceil}$, 
 $1 \leq r \leq p_1^{\left\lfloor \frac{m_1}{2}\right\rfloor} \dots p_n^{\left\lfloor \frac{m_n}{2}\right\rfloor}-1$,
 which form a clique $A_0$ in $\zt\left(\ZZ_{m}\right)$.
We will construct subsets $A_0,A_1,\ldots,A_{o(m)} $ of $V(\zt\left(\ZZ_{m}\right))$ in the 
following way. 

Let $k_1< k_2<\ldots< k_{o(m)}$ denote indices such that $m_{k_j}$ is odd for $j=1,2,\ldots,o(m)$. 
Now, define the sets $A_{j}$ to consist of all zero-divisors of the form 
$s\, p_{k_j}^{\left\lceil \frac{m_{k_j}}{2}\right\rceil-i}$, $1 \leq i \leq \left\lceil \frac{m_{k_j}}{2}\right\rceil$, 
where $s$ is coprime to $p_k$ and which are not already in $A_0,A_{1},A_{2},\ldots,A_{j-1}$. 
Note that  there are no edges within any of the sets $A_{j}$ and by definition they are disjoint. 
However, vertices
$$v_j=\frac{1}{p_{k_j}} \prod\limits_{i=1}^{n} p_i^{\left\lceil \frac{m_{k_i}}{2}\right\rceil} \in A_{j},$$
 $j=1,2,\ldots,o(m)$, form a clique and each of them is adjacent to every vertex in $A_0$. 
  Therefore, 
\begin{equation}\label{eq1}
 \chi\left(A_0 \cup A_1\cup \ldots \cup A_{o(m)}\right)\geq \chi\left(A_0 \right)+o(m)=
        \prod_{i=1}^n p_{i}^{\left\lfloor\frac{m_i}{2}\right\rfloor }+ o(m)-1.
\end{equation}


Now, let $A_{o(m)+1}=V(\zt\left(\ZZ_{m}\right))-\bigcup\limits_{i=0}^{o(m)} A_{i}$. Note that if 
$o(m)=n$, then $A_{o(m)+1}=\emptyset$. Otherwise, for every vertex 
$w=p_1^{s_1}\dots p_n^{s_n}u$ in $A_{o(m)+1}$, there exists  $s_j$ such that $m_j$ is even and 
$s_j<\frac{m_{j}}{2}$ and that $s_i \geq \left\lfloor \frac{m_{i}}{2}\right\rfloor$ if $m_i$
is odd. Clearly $w$ is not adjacent to 
 $u=p_j^{ \frac{m_{j}}{2}}\prod_{i \ne j} p_i^{m_i}\in A_0$  and
 therefore, $w$ can be colored with the same color as $u$. Notice that $w$ and $w'$ which get the same color in that way, are not adjacent to each other.
Therefore, by \eqref{eq1}
 \begin{align*}
  \chi\left(\zt(\ZZ_m)\right)= \chi\left(A_0 \cup A_1\cup \ldots \cup A_{o(m)}\right)= 
  \prod_{i=1}^n p_{i}^{\left\lfloor\frac{m_i}{2}\right\rfloor }+ o(m)-1.
  \end{align*}
  
If $\zt\left(\ZZ_{m}\right)$ is not connected, then $J \ne \emptyset$. Moreover, if $\zt\left(\ZZ_{m}\right)$ has an edge, observe that vertices of the form 
$$\prod\limits _{i \in I, t_i \geq 0}{p_i^{ \left\lfloor \frac{m_{i}}{2}\right\rfloor+t_i}} \prod\limits_{j \in J}{p_j},$$ 
together with vertices of the form $\frac{m}{p_j}$, where $j \in J$, form a clique $A$ of size 
$\prod_{i \in I}{\left\lceil \frac{m_{i}}{2}\right\rceil}+|J|$. So, at least that many colors is needed.
Let us show that coloring the remaining vertices does not require any additional colors. 

Any vertex $x=p_1^{r_1} \cdots p_n^{r_n}$ not in $A$ must have at least one of these two properties: 
\begin{enumerate}
 \item there exists $i \in I$ such that $r_i < \left\lfloor \frac{m_{i}}{2} \right\rfloor$, or 
 \item there exists $j \in J$ such that $r_j = 0$, i.e. $p_j$ does not divide $x$. 
\end{enumerate}
If $r_i < \left\lfloor \frac{m_{i}}{2} \right\rfloor$ for some $i \in I$, then $x$ is not adjacent to 
$u=mp_i^{-\left\lceil m_{i}/2 \right\rceil}\in A$, thus can be colored with the same color as $u$. In the second
case, $r_j = 0$ for some $j \in J$, then $x$ is not adjacent to $v=\frac{m}{p_j}\in A$ and can be colored
with the same color as $v$. If  
$x=p_1^{r_1} \cdots p_n^{r_n}$ and $y=p_1^{s_1} \cdots p_n^{s_n}$ are not in $A$ and would they take the 
color of the same vertex in $A$, then either there is $i \in I$ such that 
$r_i < \left\lfloor \frac{m_{i}}{2} \right\rfloor$ and $s_i < \left\lfloor \frac{m_{i}}{2} \right\rfloor$, or there is $j \in J$ such that $p_j=r_j=0$. In any case $x$ is not adjacent to $y$ and thus we found the coloring of $\zt({\ZZ_m})$
and proved that
$\chi\left(\zt\left(\ZZ_{m}\right)\right)=\prod_{i \in I} \left\lceil\frac{m_i}{2}\right\rceil + |J|$. 
\end{proof}

\begin{example}
Let us ilustrate the algorithm presented in the proof of Proposition \ref{prop:chromatic} in the case $m=2^{3}\cdot 3^{2}\cdot5=360$.

The set $A_0$ contains vertices: $2^{2}\cdot 3 \cdot 5, 2^{2}\cdot 3 \cdot 5 \cdot 2, 2^{2}\cdot 3 \cdot 5 \cdot 3, 2^{2}\cdot 3 \cdot 5 \cdot 4, 2^{2}\cdot 3 \cdot 5 \cdot 5 $. They form a $K_5$. 
The set $A_1$ is the union of the sets $ \{2^{1}, 2^{1}\cdot 3, 2^1\cdot 5, 2^1\cdot 7, 2^1\cdot 9, 2^1\cdot 11, 2^1\cdot 13, 2 ^1\cdot 15,..., 2^1\cdot 179\}$
and $\{2^{0}\cdot 3, 2^0\cdot 5,  2^0\cdot 9, 2 ^0\cdot 15,..., 2^0\cdot 177\}$. The vertex $v_1=2^1 \cdot 15$ is adjacent to every vertex in $A_0$ and $v_1\sim v_2$. 
Moreover $A_2=\{ 5^0 \cdot 4, 5^0 \cdot 8, 5^0 \cdot 12, \dots, 5^0 \cdot 356 \}$ and $v_2=2^2 \cdot 3^1 \cdot 5^0 \in A_2$ is adjacent to all the vertices in $A_0$. 
Lastly, $A_3=\{ 2^2 \cdot 3^0 \cdot 5^1, 2^3 \cdot 3^0 \cdot 5^1, 2^4 \cdot 3^0 \cdot 5^1, \cdots, 2^2 \cdot 3^0 \cdot 5^1 \cdot 17 \}$
and for every vertex in $A_3$, there is a vertex in $A_0$, which is not adjacent to it ($2^2 \cdot 3^0 \cdot 5^1  \nsim 2^2 \cdot 3^1 \cdot 5^1$, $2^3 \cdot 3^0 \cdot 5^1 \nsim 2^2 \cdot 3^1 \cdot 5^1$,\ldots). 
The graph vertices require 5 colors to color $A_0$ and all other vertices except $v_1$ and $v_2$. Since $v_1 \sim v_2$ and they are adjacent to all the vertices in $A_0$,
we have $ \chi\left(\zt\left(\ZZ_{360}\right)\right)=7$.
\end{example}

The edge coloring of a graph requires that no two adjacent edges have the same color and the smallest number of colors for the edge coloring is called the \emph{chromatic index} of $G$ and is denoted by $\chi'(G)$.

\begin{Proposition}
If $\zt\left(\ZZ_{m}\right)$ is a connected but not complete
graph, then its chromatic index is equal to $\chi'\left(\zt\left(\ZZ_{m}\right)\right)=\frac{m}{p_1}-2$.
\end{Proposition}

\begin{proof}
Recall that by Vizing's Theorem and Lemma \ref{prop:max--min--degree}  we have that $\chi'(\ZZ_{m})$ is either $\Delta\left(\ZZ_{m}\right)=\frac{m}{p_1}-2$ or $\Delta\left(\ZZ_{m}\right)+1=\frac{m}{p_1}-1$.
Let us prove that $\frac{m}{p_1}-2$ colors is sufficient to color $\zt\left(\ZZ_{m}\right)$.

By  Lemma \ref{prop:max--min--degree}, the vertices with the maximal degree are the ones associated to $\frac{m}{p_1}$.
They are of the form $s\cdot \frac{m}{p_1}$ for $s=1,2,\ldots,p_{1}-1$, hence there is $p_{1}-1$ of them.

Firstly, we color all the edges incident to $\frac{m}{p_1}$. Then, for $s=2,\ldots,p_1-1$ we color
all the uncolored edges incident to $s\frac{m}{p_1}$ in a way that for
edge $s\frac{m}{p_1}\sim v $, we take a color different from
the colors of $i \frac{m}{p_1} \sim v $ for $i=1,2,\ldots,s-1$.
For
each $v$ which is adjacent to the associatedness class of $\frac{m}{p_1}$,
edges of the form $s  \frac{m}{p_1} \sim v$ are colored
using at most $p_{1}-1$ colors. Inequality $\frac{m}{p_1}-2> p_{1}-1$
is guaranteed by Corollary \ref{prop:regular--complete}.
Coloring the rest of the edges with no additional colors is possible because each vertex not in the associatedness class of $\frac{m}{p_1}$
has degree less than $\frac{m}{p_1}-2$.
\end{proof}

\bigskip
\subsection{Cycles in $\zt\left(\ZZ_m\right)$}

To the best of our knowledge, there is no full description of rings with Hamiltonian zero-divisor graphs. In \cite{MR2553598} the authors proved that if the total graph of a 
finite commutative ring is connected then it is also a Hamiltonian graph. In the next proposition we characterize Hamiltonian total zero-divisor graphs and prove that the total zero-divisor graph is Hamiltonian if and only if it is complete with at least 4 vertices.

\begin{Proposition}\label{hamiltonian}
The total zero-divisor graph $\zt\left(\ZZ_{m}\right)$ is Hamiltonian if and only if $m=p^2$, where $p \geq 5$.
\end{Proposition}

\begin{proof}
Assume first $m=p_1^2$, where $p_1 \geq 5$. By Corollary \ref{prop:connected--iff}(3), we have $\zt\left(\ZZ_{m}\right)\cong K_{p_{1}-1}$, which is 
Hamiltonian.


If $\zt\left(\ZZ_{m}\right)$ is not connected, it is clearly not Hamiltonian. Moreover, if $m=2^2$ or $m=3^2$
then $\zt\left(\ZZ_{m}\right)$ is $K_1$ or $K_2$ and hence not Hamiltonian. So, suppose that $n\geq2$  or $m_{1}\geq3$ and note that $\frac{m}{p_1}>p_{1}$. 
For $$S=\left\{ r\cdot \frac{m}{p_1}; \; 1\leq r< p_{1} \right\} \; \text{ and} \;
T=\left\{ p_{1}s; \; 1\leq s < \frac{m}{p_1} ,\textrm{gcd}\left(s,m\right)=1\right\} 
$$
we have $\left|S\right|=p_{1}-1$. Moreover,  $r\leq p_{1}-1\leq \frac{m}{p_1}-1$ implies $\textrm{gcd}\left(r,m\right)=1$,
so $\left|S\right|\leq\left|T\right|$. Since $\textrm{gcd}\left(\frac{m}{p_1}-1,m\right)=1$
and $\frac{m}{p_1}-1 > p_{1}-1$ it follows that $\left|S\right|<\left|T\right|$. 

Notice that  vertices in $T$ have no neighbours in $V\left(\zt(\ZZ_{m})\right)-S$.
Therefore, the number of components of  $\zt\left(\ZZ_{m}\right)-S$ is at least equal to  $\left|T\right|>\left|S\right|$,
so $\zt\left(\ZZ_{m}\right)$ is not Hamiltonian.
\end{proof}

As a corollary of Lemma \ref{prop:max--min--degree} we have the following property of the total zero-divisor graph.

\begin{Corollary}
For any positive integer $m$, the total zero-divisor graph $\zt\left(\ZZ_{m}\right)$ is not Eulerian.
\end{Corollary}

\begin{proof}
If $\zt\left(\ZZ_{m}\right)$ is not connected, obviously it is not Eulerian. If $\zt\left(\ZZ_{m}\right)$ is
connected, recall that it is Eulerian if and only if every vertex has an even degree.
In the case $\zt\left(\ZZ_{m}\right)$ is complete,  by Corollary \ref{prop:connected--iff}(3), $\zt\left(\ZZ_{m}\right)\cong K_{p_{1}-1}$. Therefore,
every vertex has a degree $p_1-2$, which is either 0 or odd. Thus,  $\zt\left(\ZZ_{m}\right)$ is not Eulerian.
Otherwise, if $\zt\left(\ZZ_{m}\right)$ is connected but not complete, we make a use of Lemma \ref{prop:max--min--degree}.
If $p_{1}=2$, then $\delta\left(\zt\left(\ZZ_{m}\right)\right)=1$, and otherwise if $p_{1}>2$,  $m$ is odd and so 
$\Delta\left(\zt\left(\ZZ_{m}\right)\right)=\frac{m}{p_1}-2$ is odd as well. In both cases we have that $\zt\left(\ZZ_{m}\right)$ is not Eulerian.
\end{proof}

\bigskip
\subsection{Domination number}

Recall that a \emph{dominating set} for a graph $G$ is a subset $D$ of $V(G)$ such that every vertex not in $D$ is adjacent to at least one member of $D$. The
\emph{domination number} $\gamma(G)$ is the number of vertices in a smallest dominating set for $G$.

It was proved in \cite{MR2240381} that the domination number $\gamma(R)$ is equal to the number of distinct maximal ideals of a finite commutative ring with identity $R$, 
if $R\ne \ZZ_2 \times F$ for any field $F$ and $R$ is not a domain. The following proposition shows the same is true for the total zero-divisor graph of $\ZZ_m$.

\begin{Proposition}
If $\zt\left(\ZZ_{m}\right)$ is a connected graph, then the domination
number is equal to
$$\gamma\left(\zt\left(\ZZ_{m}\right)\right)=n.$$
\end{Proposition}

\begin{proof}
By Corollary \ref{cor:dominating--vertex}, there exists $k$, such that $v|_{m}\left(\frac{m}{p_{k}}\right)$. Hence, 
$N\left(v\right)\setminus\left\{ \frac{m}{p_{k}}\right\} \subseteq N\left(\frac{m}{p_{k}}\right)\setminus\left\{ v\right\}$. Therefore, if $v$ is in a dominating set $D$ for
$\zt\left(\ZZ_{m}\right)$, then $D\cup \left\{\frac{m}{p_{k}}\right\} - \left\{v\right\}$ is dominating set as well. Hence $\gamma\left(\zt\left(\ZZ_{m}\right)\right)\geq n$.

We claim that $D=\left\{ \frac{m}{p_{i}}; \; i=1,2,\ldots,n \right\} $
is a dominating set. Namely, for an arbitrary nonzero zero-divisor $u=p_{1}^{u_{1}}p_{2}^{u_{2}}\cdots p_{n}^{u_{n}}s_{u}$
there exists $k$ 
such that $u_{k}\geq1$. By assumption that $\zt\left(\ZZ_{m}\right)$ is connected, it follows that  $m_{k}\geq2$, thus  
$u\cdot \frac{m}{p_{k}}\equiv0$ and $p_{k}|\left(u+\frac{m}{p_{k}}\right)$. Therefore $u+\frac{m}{p_{k}}\in Z\left(\ZZ_{m}\right)$,
hence $u$ is adjacent to $\frac{m}{p_{k}}\in D$. So, $\gamma\left(\zt\left(\ZZ_{m}\right)\right)=n$.
%
%
\end{proof}

\bigskip
\subsection{Metric dimension}

The \emph{metric dimension} $\dim_M(G)$ of a graph $G$ is the minimum cardinality of a set $S \subseteq V(G)$ such that all other vertices are uniquely determined by their distances to the vertices in $S$. A set $S$ is known as a resolving set. Determing the metric dimension of a graph is known to be an NP-complete problem. 

In \cite{MR3576665} and  \cite{MR3563754}  the authors give certain bounds on the metric dimension of a zero divisor graph and of the total graph in some specific ring. 
For the total zero-divisor graph of $\ZZ_m$ we are able to compute its metric dimension and the proof allows us to apply the result also for the zero-divisor graph of $\ZZ_m$.

\begin{Proposition}
\label{prop:metric--dimension}
If  $\zt\left(\ZZ_{m}\right)$
is connected, then 
$$\dim_M(\zt\left(\ZZ_{m}\right))=
\begin{cases}
 m-\varphi\left(m\right)-\tau\left(m\right)+n+1,& n\geq 2,\\
 m-\varphi\left(m\right)-\tau\left(m\right)+1,& n=1.
\end{cases}
$$
\end{Proposition}

\begin{proof}
Recall that vertices $a$ and $b$ are indistinguishable in  a simple graph 
if 
$N\left(a\right)\setminus\left\{ b\right\} =N\left(b\right)\setminus\left\{ a\right\} $.
Note that a simple graph $G$ is invariant to permutation of indistinguishable vertices. So, complement of a resolving set of $G$ cannot contain a pair
of indistinguishable vertices because, otherwise, these would have the same ordered set of distances from the elements of resolving set.
Therefore, complement of a resolving set can contain at most one element
from each class of indistinguishable vertices. 

Assume first $n \geq 2$. Clearly, associates are indistinguishable in $\zt\left(\ZZ_{m}\right)$.
For every $i= 1,2,\ldots,n $, vertices $p_{i}^{m_{i}}$
and $p_{i}^{m_{i}-1}$ are also indistinguishable in $\zt\left(\ZZ_{m}\right)$:
clearly, $N\left(p_{i}^{m_{i}-1}\right)\setminus\left\{ p_{i}^{m_{i}}\right\} \subseteq N\left(p_{i}^{m_{i}}\right)\setminus\left\{ p_{i}^{m_{i}-1}\right\} $;
if $w\in N\left(p_{i}^{m_{i}}\right)\setminus\left\{ p_{i}^{m_{i}-1}\right\} $,
then $p_{i}|w$ because of $p_{i}^{m_{i}}+w\in Z\left(\ZZ_{m}\right)$,
and $\left(\frac{m}{p_{i}^{m_{i}}}\right)|w$ because of $p_{i}^{m_{i}}w\equiv0$,
but then $\textrm{lcm}\left(p_{i},\frac{m}{p_{i}^{m_{i}}}\right)=\frac{m}{p_{i}^{m_{i}-1}}$
is also a divisor of $w$ so $w\in N\left(p_{i}^{m_{i}}\right)\setminus\left\{ p_{i}^{m_{i}-1}\right\} $. Next we show that those two are the only types of indistinguishability in $\zt\left(\ZZ_{m}\right)$.

Let $u$ and $v$ be vertices which are neither associates nor $p_{i}^{m_{i}}$
and $p_{i}^{m_{i}-1}$ for some $i\in\left\{ 1,2,\ldots,n\right\} $.
By Lemma \ref{lem:associated--divisor}, there is no loss of generality
if we assume that $u|m$ and $v|m$. Take $u=p_{1}^{u_{1}}p_{2}^{u_{2}}\cdots p_{n}^{u_{n}}$
and $v=p_{1}^{v_{1}}p_{2}^{v_{2}}\cdots p_{n}^{v_{n}}$. If $u\neq v$
then there exists $k$ such that $u_{k}\neq v_{k}$.
Without  loss of generality assume $u_{k}<v_{k}$. If $v_{k}<m_{k}$
then $w=\frac{m}{p_k^{v_{k}}}$ is adjacent to $v$ but it is not adjacent to
$u$. Hence, $u$ and $v$ are not indistinguishable. Else if $v_{k}=m_{k}$,
then from the assumption above that $u$ and $v$ are not $p_{i}^{m_{i}}$
and $p_{i}^{m_{i}-1}$ for some $i\in\left\{ 1,2,\ldots,n\right\} $,
we get $u_{k}<m_{k}-1.$ Then $w=\frac{m}{p_k^{v_{k}-1}}$ is adjacent to $v$
but it is not adjacent to $u$. Again, $u$ and $v$ are not indistinguishable.

By the above arguments, the number of different classes of indistinguishable vertices is equal to the number of associatedness classes
minus the number of prime factors of $m$, which is $\tau\left(m\right)-2-n$ by Lemma \ref{lem:associated--divisor}.

Since the complement of a resolving set can contain at most one element from each class of indistinguishable vertices, a resolving set must have at least 
$$\left|V\left(\zt\left(\ZZ_{m}\right)\right)\right|-\left(\tau\left(m\right)-2-n\right)=
m-\varphi\left(m\right)-\tau\left(m\right)+n+1$$
elements.

Consider the set 
$$
B=\left\{ p_{1}^{b_{1}}p_{2}^{b_{2}}\cdots p_{n}^{b_{n}}s_{b}; \; p_{1}^{b_{1}}p_{2}^{b_{2}}\cdots p_{n}^{b_{n}}|m, s_{b}\in \UU\left(R\right) \setminus \{1\} \right\} \cup\left\{ p_{i}^{m_{i}}; \; i=1,2,...,n \right\} 
$$
with
$$|B|=(m-\varphi\left(m\right)-1)-(\tau\left(m\right)-2)+n=m-\varphi\left(m\right)-\tau\left(m\right)+n+1$$
and whose complement contains all the vertices corresponding to the divisors of $m$ except those in set 
$\{ p_{i}^{m_{i}}; \; 1 \leq i \leq n\}$.
We are going to show that $B$ is a resolving set. Let $x=p_{1}^{x_{1}}p_{2}^{x_{2}}\cdots p_{n}^{x_{n}}$
and $y=p_{1}^{y_{1}}p_{2}^{y_{2}}\cdots p_{n}^{y_{n}}$ be different
elements of $V\left(\zt\left(\ZZ_{m}\right)\right)\setminus B$.
Then $x$ and $y$ are not indistinguishable by definition of $B$. So, without loss of generality assume there 
exists $z\in N\left(x\right)\setminus N\left(y\right)$.
If $z\in B$, then $d\left(x,z\right)=1\neq d\left(y,z\right)$ so
ordered set of distances to elements of $B$ differ for $x$ and $y$.
Else, $z\notin B$ so $z|m$ and $z\neq p_{i}^{m_{i}}$ for any $1\leq i \leq n$.
Since $\zt\left(\ZZ_{m}\right)$ is connected and $n\geq2$, we have $m\geq36$ and thus there exists $s\in \UU\left(R\right) \setminus \{1\}$. 
Take $z'=z\cdot s\in B$ and observe that $d\left(x,z'\right)=1\neq d\left(y,z'\right)$.
So, $B$ is a resolving set and the metric dimension of $\zt\left(\ZZ_{m}\right)$
is equal to $m-\varphi\left(m\right)-\tau\left(m\right)+n+1$.

Now, let $n=1$, or equivalently $m=p_1^{m_1}$. The only indistinguishable elements are the associates since $p_1^{m_1}=0$.
It follows that   $\dim_M(\zt\left(\ZZ_{m}\right))=m-\varphi\left(m\right)-\tau\left(m\right)+1$.
\end{proof}

\begin{remark}
Note that the same proof would work to show that the metric dimension
of a connected zero-divisor graph is equal to $m-\varphi\left(m\right)-\tau\left(m\right)+1$.
\end{remark}

\bibliographystyle{plain}
\bibliography{TZDG}

\end{document}